\documentclass[11pt]
{amsart}
\usepackage[margin=1.4in]{geometry}

\usepackage{amsmath, amssymb, amsfonts, amsthm, enumerate}

\newtheorem{theorem}{\bf Theorem}[section]
\newtheorem{claim}[theorem]{\bf Claim}
\newtheorem{lemma}[theorem]{\bf Lemma}

\newtheorem{conj}[theorem]{\bf Conjecture}

\newtheorem{observation}[theorem]{\bf Observation}

\newtheorem{question}[theorem]{\bf Question}
\newcommand{\rt}{\right}
\newcommand{\lt}{\left}
\DeclareMathOperator{\Range}{Range}

\begin{document}

\title[On the maximum number of Latin transversals]{On the maximum number of Latin transversals}

\author[R.~Glebov]{Roman Glebov}
\address{Department of Mathematics, ETH, 8092 Zurich, Switzerland.}
\email{roman.l.glebov@gmail.com}
\author[Z.~Luria]{Zur Luria}
\address{Institute of Theoretical Studies, ETH, 8092 Zurich, Switzerland.}
\email{zluria@gmail.com}

\date{\today}

\begin{abstract}

Let $T(n)$ denote the maximal number of transversals in an order-$n$ Latin square.
Improving on the bounds obtained by McKay et al.,
Taranenko recently proved that $T(n) \leq \left((1+o(1))\frac{n}{e^2}\right)^{n}$,
and conjectured that this bound is tight.

We prove via a probabilistic construction that indeed $T(n) = \left((1+o(1))\frac{n}{e^2}\right)^{n}$.
Until the present paper, no superexponential lower bound for $T(n)$ was known. We also give a simpler proof of the upper bound.

\end{abstract}

\maketitle

\section{Introduction}
\label{sec:intro}

An order-$n$ {\em Latin square} is an $n \times n$ matrix $L$ over the symbols $[n]:=\{1, \ldots ,n\}$
such that every symbol appears exactly once in every row and every column.
A {\em transversal} of $L$ is a set of $n$ entries, one from every row and every column and one of every symbol.

Latin squares have been studied since ancient times, and the study of their transversals also has a rich history,
originating from the study of orthogonal Latin squares.
Two Latin squares are orthogonal to each other if each encodes a decomposition of the other into disjoint transversals.
Orthogonal pairs of Latin squares were studied by Euler, who proved that for every $n\neq 2 \mod 4$,
there exists a pair of orthogonal Latin squares of order $n$.

Much subsequent work has considered Latin transversals in the context of orthogonality. Aside from this, the literature on Latin transversals has mainly focused on two types of questions - existence and enumeration. Regarding the existence
question, it is not hard to see that Latin squares of even order do not necessarily have transversals.
Ryser~\cite{Ry67} conjectured that in every Latin square of order $n$, the number of transversals is congruent to $n \mod 2$.
Balasubramanian~\cite{Ba90} proved Ryser's conjecture for Latin squares of even order.
However, it was observed (see, e.g., \cite{AkAl04,CaWa05}) that counterexamples to Ryser's conjecture exist for Latin squares of odd order.
Subsequently, Ryser's conjecture was weakened to the statement that every Latin square of odd order has a transversal -
which remains perhaps the most intriguing open question in the study of transversals to date.
For a survey on the current state of research on Latin transversals, see~\cite{Wa11}.

Considering the enumeration problem, let $T(L)$ denote the number of transversals in a given Latin square $L$.

The {\em order-$n$ cyclic Latin square} over an alphabet $A$ is the addition table of a group over $A$ isomorphic to $(\mathbb{Z}_n,+)$.
In 1991, Vardi~\cite{Va91} conjectured the following.
\begin{conj}
\label{conj:vardi}
Let $C_n$ denote an order-$n$ cyclic Latin square. Then there exist positive constants $c_1$ and $c_2$ such that
\[ c_1 ^n n! \leq T(C_n)\leq c_2^n n!.\]
\end{conj}

Partial results on Conjecture~\ref{conj:vardi} were obtained in \cite{CoKo96,Ko96,Co00,CGKN00}.
However, a lower bound of the form conjectured above is still not known.
Based on numerical evidence, Cooper et al.~\cite{CGKN00} estimated that perhaps Conjecture~\ref{conj:vardi} is true with $c_1\sim c_2\approx 0.39$.

Let $T(n)$ denote the maximum number of transversals in an order-$n$ Latin square.
McKay, McLeod and Wanless~\cite{Mc06} showed that for constants $b \approx 1.719$ and $c \approx 0.614$,
\[
b^n\leq T(n) \leq c^n\sqrt{n} n!
\]

Improving on this, Taranenko~\cite{Ta14} recently proved a better upper bound. Before stating her result, we shall need some notation.

Let $A$ be an order-$n$ $(d+1)$-dimensional array. An {\em element} of $A$ is a tuple $(i_1, \ldots , i_{d+1}) \in [n]^{d+1}$.
A {\em $\lambda$-element} of $A$ is a tuple $(i_1, \ldots , i_{d+1})$ such that $A(i_1, \ldots ,i_{d+1})=\lambda$.
A {\em line} of $A$ is a set of elements obtained by fixing all but one of the indices and allowing the free index to vary over $[n]$.
A {\em hyperplane} of $A$ is a set of elements obtained by fixing one index and allowing the remaining $d$ indices to vary over $[n]$.

An {\em order-$n$ $d$-dimensional Latin hypercube} is an $[n]^{d}$ array over $[n]$ such that every symbol appears exactly once in every line.
In the same way that a permutation may be represented by a permutation matrix, an order-$n$ $d$-dimensional Latin hypercube is equivalent to an $[n]^{d+1}$ 0-1 array with exactly one $1$ in each line, and in what follows we use this definition.

We generalize the notion of a transversal to the high dimensional case. A transversal of $A$ is a set of $n$ $1$-elements of $A$, exactly one from each hyperplane. In the case of Latin squares, this is equivalent to our previous definition.
Let $T(d,n)$ denote the maximum number of transversals in an order-$n$ $d$-dimensional hypercube.

\begin{theorem}[Theorem 6.1 in~\cite{Ta14}]
\label{upperbound}
For every $d\geq 2$ and $n\rightarrow \infty$,
\[
T(d,n) \leq \left((1+o(1))\frac{n^{d-1}}{e^d}\right)^{n}.
\]
\end{theorem}

In particular, a Latin square has at most $\left((1+o(1))\frac{n}{e^2}\right)^{n}$ transversals.
This implies the upper bound $T(C_n)\leq \lt((1+o(1))e^{-1}\rt)^n n!$,
slightly below the numerical estimate $0.39^n n!$ of Cooper et al.

The main contribution of the present paper is proving that this bound is asymptotically tight for all $n$ in the case of Latin squares.
As we will discuss in Section~\ref{sec:concl},
it is also tight for all $d$ and infinitely many values of $n$ in the case of general $d$-dimensional Latin hypercubes.

\begin{theorem}
\label{lowerbound}
For every $n$,
\[
T(n) \geq \left((1-o(1))\frac{n}{e^2}\right)^{n}.
\]
\end{theorem}

Our paper is organized as follows.
In Section~\ref{sec:upper} we present a short alternative proof of the upper bound from Theorem~\ref{upperbound}.
In Section~\ref{sec:lower} we give a probabilistic construction of an order-$n$ Latin square
with at least $\left((1-o(1))\frac{n}{e^2}\right)^{n}$ transversals, proving Theorem~\ref{lowerbound}.
Finally, in Section~\ref{sec:concl} we discuss possible generalizations of these results and related open questions.

\section{A short proof of Theorem~\ref{upperbound}}
\label{sec:upper}
In this section, we present a short proof of Theorem~\ref{upperbound} based on the entropy method.
This method has recently emerged as a powerful tool for a variety of counting problems.
It has been extensively used to obtain upper bounds
(see, e.g. \cite{Ra97, Ka01, LiLu13, LiLu14, Ke15+}), but it has also been used for lower bounds,
as in Johanssen, Khan and Vu's seminal paper~\cite{JKV08} on the threshold for matchings in hypergraphs.

The basic idea is as follows. In order to estimate the cardinality of a set $S$, we sample an element $X \in S$ at random. Since $H(X) = \log(|S|)$, bounds on $H(X)$ translate into bounds on $|S|$. This is of use because there are powerful information theoretic tools for estimating entropies of random variables, and so estimating $H(X)$ is sometimes easier than estimating $|S|$ directly.

Below we summarise the properties of the entropy function that we use.
\begin{enumerate} [$(i)$]
\item The base $e$ entropy of a random variable $X$ is
\[
H(X) = \sum_{x \in \Range(X)}{\Pr(X=x)\log\left(\frac{1}{\Pr(X=x)}\right)}.
\]
The entropy of a discrete random variable can be interpreted as being the amount of information it encodes.
\item
For any discrete random variable $X$, $H(X) \leq \log(|\Range(X)|)$, with equality if and only if $X$ is uniformly distributed.
\item $H(X|Y=y)$ is the entropy of $X$ conditioned on the event $Y=y$. The conditional entropy $H(X|Y)$ is
\[
H(X|Y) = \sum_{y \in \Range(Y)}{\Pr(Y=y) H(X|Y=y)} = \mathbb{E}_Y[H(X|Y=y)].
\]
\item The chain rule: If $X = (X_1, \ldots ,X_n)$ then for any ordering of $[n]$,
\[
H(X) = \sum_{i=1}^n{H(X_i|X_j \text{ such that $j$ precedes $i$})}.
\]
\end{enumerate}

Let $A$ be an order-$n$ $d$-dimensional Latin hypercube in 0-1 representation,
and let $X$ be a transversal of $A$ selected uniformly at random,
so that $e^{H(X)}$ is the number of transversals in $A$.
Let $X_i$ denote the element of $A$ chosen for $X$ from the hyperplane $A_i := A(i,\cdot,\ldots,\cdot)$.
We choose a random ordering of these variables by selecting a number $\alpha_i \in [0,1]$ uniformly at random for each such hyperplane,
and let $\alpha = \{\alpha_i\}_{i=1}^n$. The variables $X_i$ are exposed in order of decreasing $\alpha_i$. Now,
\begin{align}
\log(T(A)) &= H(X)\nonumber\\
&\overset{(iv)}{=} \mathbb{E}_{\alpha}\left[ \sum_{i=1}^n{H(X_i|X_j \text{ such that $\alpha_j > \alpha_i$})}\right]\nonumber\\
&\overset{(iii)}{=} \mathbb{E}_{\alpha}\left[ \sum_{i=1}^n{\mathbb{E}_{X_j:\alpha_j>\alpha_i} \left[ H(X_i|X_j=x_j \text{ for $j$ such that $\alpha_j > \alpha_i$})\right] } \right]\nonumber\\
&= \mathbb{E}_{\alpha}\left[ \sum_{i=1}^n{\mathbb{E}_{X} \left[ H(X_i|X_j=x_j \text{ for $j$ such that $\alpha_j > \alpha_i$})\right] } \right]\nonumber\\
&= \sum_{i=1}^n{\mathbb{E}_{X}\left[\mathbb{E}_{\alpha}\left[ H(X_i|X_j=x_j \text{ for $j$ such that $\alpha_j > \alpha_i$}) \right]\right]},
\label{formula1}
\end{align}
where the last two equalities hold by basic properties of expectation.

We say that two elements $(i_1,\ldots,i_{d+1})$ and $(j_1,\ldots,j_{d+1})$ share an index if there is a $k$ such that $i_k=j_k$.
Given previous variables, $X_i$ cannot choose an element that belongs to the same hyperplane as a previousy seen variable. A $1$-element in $A_i$ is {\em legal} if it does not share an index with a previously seen $X_j$. Let $N_i = N_i(\alpha,X)$ be the number of legal elements in $A_i$.

We cannot say very much about $X_i$'s distribution given previously observed variables, but $X_i$ certainly must choose a legal element and therefore $|\Range(X_i)|\leq N_i$.
From $(i)$ and by applying Jensen's inequality we obtain
\begin{align}
 \mathbb{E}_{\alpha}[H(X_i|X_j=x_j \text{ for $j$ such that $\alpha_j > \alpha_i$})] &\leq \mathbb{E}_{\alpha}[ \log(N_i)]\nonumber\\
 &= \mathbb{E}_{\alpha_i}[ \mathbb{E}_{\alpha_j:j \ne i}[\log(N_i)]] \nonumber\\
 &\leq  \mathbb{E}_{\alpha_i}[ \log(\mathbb{E}_{\alpha_j:j \ne i}[N_i])].
 \label{formula2}
\end{align}

The next step is to compute $ \mathbb{E}_{\alpha_j:j \ne i}[N_i] $ using linearity of expectation.
Conditioning on $X$ and $\alpha_i$, what is the probability that a $1$-element of the $i$-th hyperplane is legal?

There is always one $1$-element, the one chosen by $X_i$, that is legal for $X_i$ regardless of the ordering.
To compute the probability that a different $1$-element of the $i$-th hyperplane is legal,
we must count the number of different variables $X_j$ that rule out that element for $X_i$.

Here we have to be a little careful. In the two dimensional case, every $1$-element in $A_i$ except $X_i$ is ruled out by exactly two variables. In the general $d$-dimensional case, every $1$-element is ruled out by {\em at most} $d$ variables. Our task is to show that a typical $1$-element is ruled out by exactly $d$ variables.

Let $V$ denote the set of $1$-elements in $A_i$.
Let $U$ denote the set of $1$-elements of $A_i$ that share at least two indices with an element of $X$.
The following claim provides an upper bound on the size of $U$.

\begin{claim}
\label{cla:1}
For the set $U$ as described above, $|U|\leq d^2\cdot n^{d-2}$.
\end{claim}

\begin{proof}
Let $X_i=(i, j^*_1,\cdot, j^*_d)$.
For every $k\in [d]$, there are exactly $n^{d-2}$ $1$-elements of $A_i$ that share $j^*_k$ with $X_i$.
This accounts for at most $d\cdot n^{d-2}$ $1$-elements of $U$.

The second possibility is that a $1$-element of $A_i$ shares two or more indices with some $X_{i'}=(i', j'_1,\cdot, j'_d)$ for $i'\ne i$.
If we fix two coordinates $m,\ell \in [d]$, there are exactly $n^{d-3}$ $1$-elements of $A_i$ that share the indices $j'_{m}$ and $j'_{\ell}$ with $X_{i'}$.
There are $n-1$ possibilities to choose the index $i'$ and $\binom{d}{2}$ possibilities for $m$ and $\ell$,
so this accounts for at most $(n-1)\binom{d}{2} n^{d-3}$ elements of $U$.
Summing up the two values, the desired upper bound follows.
\end{proof}

The next claim states that for every element in $V\setminus U$, there exist exactly $d$ elements of $X$
that rule it out.

\begin{claim}
\label{cla:2}
For every $v=(i,j_1, \ldots ,j_d) \in V\setminus U$,
there exist exactly $d$ indices $i_1, \ldots, i_d$,
all different from each other and from $i$, such that $X_{i_k}$ rules out $v$.
\end{claim}

\begin{proof}
Fix a coordinate $k\in [d]$.
Since $X$ is a transversal, there is a unique tuple $\lt(i_k,j_1^{(k)}, \ldots ,j_d^{(k)}\rt)$
such that $j_k^{(k)} = j_k$ and $X\lt(i_k,j_1^{(k)}, \ldots ,j_d^{(k)}\rt)=1$.
Thus, if $X_{i_k}$ precedes $X_i$ then by the time we see $X_i$
we already know of a $1$-element of $X$ that shares an index with $v$, and therefore $v$ is ruled out.

Note that $i_k \ne i$ because otherwise $v$ would share at least two indices, $i=i_k$ and $j_k=j_k^{(k)}$,
with $X_i$. Since $v \notin U$, $v$ can share at most one index with an element of $X$.

Similarly, $i_{m} \ne i_{\ell} $  for all distinct $m,\ell\in [d]$
because otherwise, $i_{m} = i_{\ell} $ implies
$\lt(i_{m},j_1^{(m)}, \ldots ,j_d^{(m)}\rt)=\lt(i_{\ell},j_1^{(\ell)}, \ldots ,j_d^{(\ell)}\rt)$, which means that $v$ shares at least two indices,
$j_m=j_m^{(m)}=j^{(\ell)}_{m}$ and $j_\ell=j_{\ell}^{(m)}=j_\ell^{(\ell)}$, with $\lt(i_{m},j_1^{(m)}, \ldots ,j_d^{(m)}\rt)$.
This contradiction provides the statement of the claim.
\end{proof}

Summarizing Claim~\ref{cla:1} and Claim~\ref{cla:2}, we obtain
\begin{align}
\mathbb{E}_{\alpha_j:j \ne i}[N_i] & = \sum_{v\in V} \Pr (v\mbox{ is legal})\nonumber\\
& =  \sum_{v\in U} \Pr (v\mbox{ is legal}) + \sum_{v\in V\setminus U} \Pr (v\mbox{ is legal})\nonumber\\
& \leq |U| + \lt(|V\setminus U|\rt) \alpha_i^d\nonumber\\
&\leq d^2\cdot n^{d-2} + \left(1 - \frac{d^2}{n}\right)n^{d-1}\alpha_i^d,
\label{formula3}
\end{align}
where the factor $\alpha_i^d$ in the second summand is the probability that for all $d$ indices guaranteed by Claim~\ref{cla:1},
the value of $\alpha$ is smaller than $\alpha_i$,
and therefore a particular element $v\in V\setminus U$ precedes all $d$ elements that could have ruled it out.

Now we are ready to put together the statements $\eqref{formula1},\eqref{formula2}$, and $\eqref{formula3}$:
\begin{align*}
\log(T(A)) &\leq \sum_{i=1}^n{\mathbb{E}_{X}\left[\mathbb{E}_{\alpha_i}
\left[\log \lt( d^2\cdot n^{d-2} + \left(1 - \frac{d^2}{n}\right)n^{d-1}\alpha_i^d\rt) \right]\right]}\\
& \leq n \int_0^1\log \lt( d^2\cdot n^{d-2} + \left(1 - \frac{d^2}{n}\right)n^{d-1}\alpha_i^d\rt) d\alpha_i\\
& = n\left(\log(n^{d-1})-d+o(1)\right)).
\end{align*}
Thus,
\[
T(A) \leq \left((1+o(1))\frac{n^{d-1}}{e^d}\right)^n,
\]
providing the theorem.

\section{Proof of Theorem~\ref{lowerbound}}
\label{sec:lower}
In this section, we present a probabilistic construction of a Latin square $L'=L'(N')$ of order $N'$, for every sufficiently large integer $N'$,
such that the number of transversals in $L'$ is $\left((1-o(1))\frac{N'}{e^2}\right)^{N'}$.

In what follows, we think of an order-$n$ Latin square $L$ as an $n\times n$ matrix over an alphabet $A$ of cardinality $n$.
The index pairs of $L$ are referred to as {\em positions}, and elements of the alphabet are the {\em symbols} of $L$.

Let us choose an integer $ b\geq 3$ such that with $n=b^2$, $k=\lfloor b-b^{9/10}\rfloor$, and $N=n^2 k$,
we obtain $N\leq N'-3$ as large as possible.
Furthermore, denote $\ell=n-bk$.
Our plan is as follows: first we construct a (probabilistic) Latin square $L$ of order $N$ with many transversals.
Then in the next step, we (deterministically) construct the desired order-$N'$ Latin square $L'$ from $L$,
with $N'$ between $N+3$ and $N + \ell n k=N+\Theta\left(N^{49/50}\right)$,
without decreasing the number of transversals significantly.

For the construction of $L$, we first take an order-$n$ Latin square $S$ over the alphabet $[n]$ that will define the {\em structure} of $L$.
We require $S$ to have a collection of $\ell$ disjoint transversals.
This is easy to achieve, as for every $n\geq 7$ there exists a pair of orthogonal Latin squares (see, e.g., \cite{BSP60}),
and so each of them has a decomposition into disjoint transversals.

Let $A_1,\ldots,A_n$ be a partition of $[N]$ into $n$ alphabets, each of size $nk$.
Let $S_1, \ldots, S_n$ be $n$ order-$(nk)$ Latin squares where the symbols of $S_i$ belong to the alphabet $A_i$,
and each of these squares has a decomposition into disjoint transversals. The existence of such squares was established in the previous paragraph.

Our desired Latin square $L$ is a block matrix, where each block is an order-$(nk)$ Latin square that corresponds to a position of $S$.
To construct $L$, for every position $(i,j)\in [n]^2$, we replace the symbol $s=S(i,j)$ in $S$ with a Latin square over the alphabet $A_{s}$.

Fix $\ell$ disjoint transversals in $S$, and call their positions {\em special}.
For every $i\in [n]$, we replace every symbol $i$ in a special position with the square $S_i$,
and every symbol $i$ in a non-special position with a random Latin square over the alphabet $A_i$,
chosen independently and uniformly at random from the set of such Latin squares.
We also call the positions of $L$ from the subsquares corresponding to special positions of $S$ {\em special},
and these subsquares themselves are also referred to as {\em special} subsquares.
In the future, whenever we refer to {\em subsquares} of $L$, we mean the $n^2$ Latin squares that we used to replace symbols of $S$.

We need the following two statements about $L$. The first statement is deterministic, as it only deals with special positions of $L$.
We will only use it to construct the slightly larger Latin square $L'$ from $L$.

\begin{observation}
\label{obs:special}
The Latin square $L$ contains a collection of $\ell nk$ disjoint transversals containing only special positions.
\end{observation}

The second statement gives a probabilistic lower bound for the actual count of transversals in $L$, and later in $L'$.

\begin{lemma}
\label{lem:non-special}
The Latin square $L$ has in expectation at least $\left((1-o(1))\frac{N}{e^2}\right)^N$ transversals not containing any special positions.
\end{lemma}

\begin{proof}
We consider only transversals that contain exactly $b$ positions from every non-special subsquare of $L$.

First, we bound the number of ``potential non-special transversals'', or in other words permutation matrices containing $b$ positions from every non-special subsquare.
To choose the $b$ positions in the first non-special subsquare, we have $\binom{kn}{b}$ possibilities to determine the rows,
and the same number of possibilities for the columns, and we should multiply the product by $b!$,
since we are not interested in the order in which the $b$ positions are chosen.
We end up with $\left(kn(kn-1)\cdots (kn-b+1)\right)^2/b!$ choices.

Similarly, dealing with non-special subsquares one by one and choosing all $b$ positions in each of them before continuing to the next subsquare,
assume that we arrive at a (non-special) subsquare
such that we already have chosen $b$ positions from $i$ subsquares in the same row and $j$ subsquares from the same column of $S$.
To choose $b$ positions in the current subsquare, we need to choose $b$ rows out of $kn-ib$ possible rows,
and $b$ columns out of $kn-jb$ possible columns,
and combine them to get $b$ positions.
This results in a total number of
\[\left((kn-ib)\cdots (kn-ib-b+1)\right)\cdot \left((kn-jb)\cdots (kn-jb-b+1)\right)/b!\]
possible choices.

Notice crucially that independently of the order in which we are dealing with the subsquares,
for every $i< n-\ell$, there will be exactly $n-\ell$ non-special subsquares
for which we dealt with exactly $i$ subsquares in the same row of $S$ before,
and similarly exactly $n-\ell$ times
we will deal with a non-special subsquare for which we dealt with exactly $i$ subsquares in the same column of $S$ before.
We end up with a total number of
\[\frac{(kn!)^{2n}}{(b!)^{n(n-\ell)}}=n^{O(1)}\frac{(kn)^{2n^2k}}{e^{2n^2k}}\cdot \frac{e^{bn(n-\ell)}}{b^{bn(n-\ell)}}=(1-o(1))^N\frac{N^N}{e^{N}}
\]
such permutation matrices.

Now we need to estimate the probability $P$ for an arbitrary such permutation matrix to correspond to an actual transversal in $L$.
Notice that we are not aiming for the best possible estimations,
but just for those that will suffice for our purposes.
Since we already made sure that there is exactly one position from every row and every column,
there are just two bad events left:
\begin{itemize}
\item
$E_1$: Two ones of the permutation matrix from different subsquares of $L$ correspond to the same symbol.
\item
$E_2$: Two ones of the permutation matrix from the same subsquare of $L$ correspond to the same symbol.
\end{itemize}

We interpret the random process of creating $L$ from $S$ as follows.
We replace non-special entries of $S$ by subsquares one-by-one, and check whether $E_1$ or $E_2$ occurred, after every replacement.

We consider $E_2$ first.
Formally, we have $b$ fixed positions, draw a random Latin square over a given alphabet, and are wondering whether in any two of the chosen positions,
the symbols are the same.
But since permuting the rows and columns of a random Latin square does not change the probability space,
we can assume that we are given an {\em arbitrary} Latin square, and chose $b$ {\em random} positions in it from different rows and columns.
Notice that it is not important now what choices we made in other subsquares, neither the positions (their rows and columns) nor the symbols there!
For a fixed non-special subsquare of $L$ that we are now looking at, assume that we already
chose some number $0\leq i < b$ of positions
in different rows and different columns, and all corresponding symbols were different -
and we are now choosing a random $(i+1)$-st position from the remaining $(kn-i)\times(kn-i)$-square.
To estimate the desired probability, observe that every symbol from one of the already chosen positions occurs at most $kn-i$ times
in the remaining $(kn-i)\times(kn-i)$-square.
Therefore, the probability that the symbol in the $(i+1)$-st chosen position is different from the symbols in the previously chosen $i$ positions
is at least $1-\frac{i}{kn-i}=1-o(1)$.
Hence, the probability that the symbols we chose from the above considered subsquare of $L$
were pairwise distinct
is $e^{o(b)}$.

For the probability of $(E_1|E_2)$,
assume that we dealt with $0\leq j<n-\ell$ subsquares over the same alphabet as $L$ and chose $b$ positions with distinct symbols from each of them,
and are now dealing with the next subsquare (over the same alphabet).
Furthermore,
since we can also permute the symbols of the alphabet without changing the probability space of a random Latin square,
we see that,
conditioned on the event that all $b$ chosen symbols were pairwise distinct,
the probability that all of them were different from the $jb$ previously chosen symbols from other subsquares over the same alphabet
is exactly
\[\binom{kn-jb}{b}/\binom{kn}{b}.\]
Thus, we obtain
\[P \geq e^{o(N)}\left(\prod_{j<n-\ell}\binom{kn-jb}{b}/\binom{kn}{b}\right)^n \geq  e^{o(N)}\left((kn)!/(kn)^{kn}\right)^n= (1-o(1))^N/{e^{N}} .\]

Therefore, the expected number of transversals in $L$ not containing any special position is at least
\[(1-o(1))^N\frac{N^N}{e^{N}}\cdot (1-o(1))^N/{e^{N}},
\]
and the lemma follows.
\end{proof}
In particular, there exists an order-$N$ Latin square $L$ with at least $\left((1-o(1))\frac{N}{e^2}\right)^N$ transversals not containing any special positions, and satisfying the statements of Observation~\ref{obs:special}.

We are now ready to prove Theorem~\ref{upperbound}.
Based on the Latin square $L$, we construct a Latin square $L'$ of order $N'$.

Let us take an alphabet $A^*$ of size $s:=N'-N$ that is disjoint from the alphabet of $L$ - we assume w.l.o.g. that $A^*=[s]$.
We put in the upper left $N\times N$-corner of $L'$ a copy of $L$,
and in the bottom right $s\times s$-corner of $L'$ an arbitrary Latin square $L^*$ over the alphabet $A^*$
containing at east one transversal.
(Notice that the existence of such Latin square $L^*$ is trivial since $s>2$.)
Finally, in $L'$, we replace $s$ of the special transversals of $L$ by symbols from $A^*$ as follows.
One after another, we take a symbol $i\in A^*$ and a special transversal from $L$ that we have not dealt with yet.
For every position $(x,y)$ from this transversal,
we put $i$ into the position $(x,y)$, and the symbol that was originally in the position $(x,y)$ in $L$,
we move into the positions $(x,N+y)$ and $(N+x, y)$ in $L'$.
The resulting $L'$ is therefore a Latin square,
and the number of transversals in $L'$ is at least as large as the number of transversals in $L$ consisting only of non-special positions,
proving the theorem.

The only condition on the relation between $N$ and $N'$ that needs to be satisfied is
$N'\in \{N+3, \ldots, N+\ell n\}$ - and this is easy to achieve, since increasing $b$ by one in the construction of $L$
increases $N$ by at most $2N^{1/5}$,
and $\ell n =\Theta\left(N^{49/50}\right)$.

\section{Concluding remarks and open questions}
\label{sec:concl}

Similarly to Theorem~\ref{lowerbound}, and keeping in mind the upper bound from Theorem~\ref{upperbound},
it is natural to consider the lower bound in the higher dimensional case.
Analogously to the proof of Lemma~\ref{lem:non-special}, choosing parameters $b,~n,~k$ as in that proof
and $N=kn^d$, we can prove the following.

\begin{theorem}
\label{lowerbound2}
For every $d \geq 3$ and for infinitely many values of $N$,
\[
T(d,N) \geq \left((1-o(1))\frac{N^{d-1}}{e^d}\right)^{N}.
\]
\end{theorem}
We strongly believe that, similarly to Latin squares, this statement is true for {\em every} value of $N$.

Another natural question concerns the distribution of $T(L)$.
Our result implies that for a certain collection of Latin squares,
the number of transversals is $(1-o(1))^n\cdot T(n)$.
It is natural to wonder whether the much more general statement is true
and $T(L)=(1-o(1))^n \cdot T(n)$ for almost every order-$n$ Latin square $L$?

Finally, analogously to the question about $T(n)$,
let $t(n)$ denote the {\em minimum} number of transversals over all order-$n$ Latin squares. For even $n$ it is known that $t(n)=0$.
As mentioned in the introduction, the most important open problem in the study of transversals is Ryser's conjecture, which states that $t(n)>0$ for odd $n$.
Can we strengthen the previous question about the distribution of $T(L)$ in the case of Latin squares of odd order?

\begin{question}
\label{q:t=T}
Is it true that $t(n)=(1-o(1))^n\cdot T(n)$ for odd $n$?
\end{question}


\begin{thebibliography}{99}

\bibitem{AkAl04}
S. Akbari and A. Alireza, Transversals and multicolored matchings, {\em J. Combin. Des.} {\bf 12} (2004), 325--332.

\bibitem{Ba90}
K. Balasubramanian, On transversals in Latin squares, {\em Linear Algebra Appl.} {\bf 131} (1990), 125--129.

\bibitem{BSP60}
R.C. Bose, S.S. Shrikhande, and E.T. Parker,
Further results on the constructive of mutually orthogonal Latin squares and the falsity of Euler's conjecture,
{\em Canad. J. Math.} {\bf 12} (1960), 189--203.

\bibitem{CaWa05}
P.J. Cameron and I.M. Wanless, Covering radius for sets of permutations,
{\em Discrete Math.} {\bf 293} (2005), 91--109.

\bibitem{Co00}
C. Cooper, A lower bound for the number of good permutations,
{\em Data Recording, Storage and Processing (Nat. Acad. Sci. Ukraine)}
{\bf 213} (2000), 15--25.

\bibitem{CGKN00}
C. Cooper, R. Gilchrist, I. Kovalenko, and D. Novakovic, Deriving the
number of good permutations, with applications to cryptography, {\em Cybernet. Systems Anal.} {\bf 5} (2000), 10--16.

\bibitem{CoKo96}
C. Cooper and I. Kovalenko, The upper bound for the number of
complete mappings, {\em Theory Probab. Math. Statist.} {\bf 53} (1996), 77--83.

\bibitem{JKV08}
A. Johansson, J. Kahn, and V. Vu, Factors in random graphs,
{\em Random Structures Algorithms} {\bf 33} (2008), 1--28.

\bibitem{Ka01}
J. Kahn, An entropy approach to the hard-core model on bipartite graphs, {\em Combin. Probab. Comput.} {\bf 10} (2001), 219--237.

\bibitem{Ke15+}
P. Keevash, Counting designs, preprint.

\bibitem{Ko96}
I. Kovalenko, On an upper bound for the number of complete mappings, {\em Cybernet. Systems Anal.} {\bf 32} (1996), 65--68.

\bibitem{LiLu14}
N. Linial and Z. Luria, An upper bound on the number of high-dimensional permutations,
{\em Combinatorica} {\bf 34} (2014), 471--486.

\bibitem{LiLu13}
N. Linial and Z. Luria, Upper bounds on the number of Steiner triple systems and $1$-factorizations,
{\em Random Structures Algorithms} {\bf 43} (2013), 399--406.

\bibitem{Mc06}
B.D. McKay, J.C. McLeod, and I.M. Wanless, The number of transversals in a Latin square,
{\em Des. Codes Cryptogr.} {\bf 40} (2006), 269--284.

\bibitem{Ra97}
J. Radhakrishnan, An entropy proof of Br\'egman's theorem, {\em J. Combin. Theory Ser. A} {\bf 77} (1997), 161--164.

\bibitem{Ry67}
H.J. Ryser, Neuere Probleme der Kombinatorik,
{\em Vortr\"age \"uber Kombinatorik, Oberwolfach}, 24-29 Juli (1967), 69--91.

\bibitem{Ta14}
A.A. Taranenko, Multidimensional Permanents and an Upper Bound on the Number of Transversals in Latin Squares,
{\em J. Combin. Des.} {\bf 23} (2014), 305--320.

\bibitem{Va91}
I. Vardi, {\bf Computational Recreations in Mathematics}, Addison-Wesley,
Redwood City, CA, 1991.

\bibitem{Wa11}
I.M. Wanless, Transversals in Latin squares: A survey, in R. Chapman (ed.),
{\bf Surveys in Combinatorics}, London Math. Soc. Lecture Note Series 392, Cambridge University
Press, 2011, 403-–437.


\end{thebibliography}
\end{document}